\documentclass[11pt]{article}

\usepackage{fullpage}
\usepackage{authblk}
\usepackage{natbib}
\usepackage{algorithm}
\usepackage{algpseudocode}
\usepackage{amsmath,amsthm,amsfonts}
\usepackage{amsmath}
\usepackage[subnum]{cases}
\usepackage{tcolorbox}
\usepackage{mathrsfs}
\usepackage{amssymb}
\usepackage{color}
\usepackage{mathrsfs}
\usepackage{enumitem}
\usepackage{bm}
\usepackage{multirow}
\usepackage{makecell}
\usepackage{graphicx}
\usepackage{subcaption}
\usepackage{comment}
\usepackage{cases}
\usepackage{appendix}
\usepackage{tikz}
\usetikzlibrary{arrows,shapes}
\usepackage[colorlinks= true, linkcolor=red, citecolor=blue, urlcolor=black]{hyperref}
\usepackage{cleveref}
\usepackage{marginnote}
\usepackage{diagbox} 

\numberwithin{equation}{section}

\newtheorem{theorem}{Theorem}[section]
\newtheorem{lemma}[theorem]{Lemma}
\newtheorem{corollary}[theorem]{Corollary}

\newcommand{\mr}{\mathbb{R}}

\newcommand{\norm}[1]{\left\|#1\right\|}
\newcommand{\inner}[1]{\left\langle #1\right\rangle}

\newcommand{\E}{\mathcal{E}}

\DeclareMathOperator*{\argmin}{arg\,min}

\usepackage{booktabs}
\usepackage{pgfplots}
\usetikzlibrary{calc,intersections}

\usetikzlibrary{shapes.geometric, arrows, positioning}

\tikzset{
	mybox/.style  = {draw, rectangle, minimum width=4cm, minimum height=0.8cm, text centered, text width=4.4cm,   
		font=\normalsize},
	box/.style  = {draw, rectangle, minimum width=2.0cm, minimum height=0.6cm, text centered, text width=3.0cm,   
		font=\normalsize},
	myarrow/.style = {line width=0.2pt, draw=black, -triangle 60, postaction={draw, line width=0.2pt, shorten >=10pt,-}}
}

\tikzstyle{arrow} = [->, >=stealth, -triangle 60]

\allowdisplaybreaks

\makeatletter
\newcommand{\leqnomode}{\tagsleft@true}
\newcommand{\reqnomode}{\tagsleft@false}
\makeatother

\begin{document}

\title{Proximal Subgradient Norm Minimization of ISTA and FISTA\thanks{This work was partially supported by grant 12288201 of NSF of China.}}

\author[1,2]{Bowen Li\qquad Bin Shi\thanks{Corresponding author, Email: \url{shibin@lsec.cc.ac.cn}} \qquad Ya-xiang Yuan} 

\affil[1]{State Key Laboratory of Scientific and Engineering Computing, Academy of Mathematics and Systems Science, Chinese Academy of Sciences, Beijing 100190, China  }

\affil[2]{University of Chinese Academy of Sciences, Beijing 100049, China}

\date\today

\maketitle

\begin{abstract}
For first-order smooth optimization, the research on the acceleration phenomenon has a long-time history. Until recently, the mechanism leading to acceleration was not successfully uncovered by the gradient correction term and its equivalent implicit-velocity form. Furthermore, based on the high-resolution differential equation framework with the corresponding emerging techniques, phase-space representation and Lyapunov function, the squared gradient norm of Nesterov's accelerated gradient descent (\texttt{NAG}) method at an inverse cubic rate is discovered. However, this result cannot be directly generalized to composite optimization widely used in practice, e.g., the linear inverse problem with sparse representation. In this paper, we meticulously observe a pivotal inequality used in composite optimization about the step size $s$ and the Lipschitz constant $L$ and find that it can be improved tighter. We apply the tighter inequality discovered in the well-constructed Lyapunov function and then obtain the proximal subgradient norm minimization by the phase-space representation, regardless of gradient-correction or implicit-velocity. Furthermore, we demonstrate that the squared proximal subgradient norm for the class of iterative shrinkage-thresholding algorithms (ISTA) converges at an inverse square rate, and the squared proximal subgradient norm for the class of faster iterative shrinkage-thresholding algorithms (FISTA) is accelerated to convergence at an inverse cubic rate.    
\end{abstract}



\section{Introduction}
\label{sec: intro}


Since entering the new century, we have witnessed the rapid development of machine learning, where the heart is the design of efficient algorithms. The major problem is that the computation and storage of Hessians are often infeasible in large-scale optimization. Hence, according to the cheap computation, simple gradient-based optimization algorithms have become the workhorse powering recent developments. Recall that the smooth convex unconstrained optimization problem is formulated as 
\[
\min_{x\in \mathbb{R}^d}f(x).
\]
Perhaps the simplest first-order method is the vanilla gradient descent, which can be traced back to~\citet{cauchy1847methode}. Taking a fixed step size $s>0$, the vanilla gradient descent is implemented as the recursive rule
\[
x_{k} = x_{k-1} - s\nabla f(x_{k-1}),
\]
given an initial point $x_0 \in \mathbb{R}^d$ and with the global convergence rate as
\begin{equation}
\label{eqn:1-over-k}
f(x_k) - f(x^\star) \leq O\left( \frac{1}{sk}\right).
\end{equation}
The modern idea of acceleration begins with~\citet{polyak1964some}, where his heavy-ball method (or momentum method) is proposed to locally accelerate the convergence rate based on the invariant manifold theorem from the field of dynamical systems. Then, one of the milestones is Nesterov's accelerated gradient descent (\texttt{NAG})
\[
\left\{\begin{aligned}
       & x_{k} = y_{k-1} - s\nabla f(y_{k-1}),              \\
       & y_{k} = x_{k}  + \frac{k-1}{k+r}(x_{k} - x_{k-1}),
       \end{aligned}\right.
\]   
with an initial point $x_0 = y_{0} \in \mathbb{R}^d$, of which the case $r=2$ is originally proposed by~\citet{nesterov1983method}. And the convergence rate is also globally accelerated to be\footnote{Actually, for the case $r>2$, the faster convergence rate has been discovered in~\citep{attouch2016rate} and~\citep{chen2022gradient}, respectively as
\[
f(x_k) - f(x^\star) \leq o\left( \frac{1}{sk^2}\right), \quad \text{or} \quad
\min_{0\leq i \leq k}f(x_i) - f(x^\star) \leq o\left( \frac{1}{sk^2}\right)
\]
} 
\begin{equation}
\label{eqn:1-over-k2}
f(x_k) - f(x^\star) \leq O\left( \frac{1}{sk^2}\right).
\end{equation}

\paragraph{Composite optimization} Although the deterministic gradient-based algorithms have obtained theoretical guarantees for smooth optimization, the objective function used in practice is often nonsmooth. More specifically, a class of objective functions with wide applications is the composite of two functions 
\begin{equation}
\label{eqn: composite}
\min_{x\in\mathbb{R}^d}\Phi(x) := f(x) + g(x),
\end{equation}
where $f$ is a smooth convex function and $g$ is a continuous convex function that is possibly not smooth. If $g$ is zero, the composite objective function $\Phi$ degenerates to the smooth function $f$. In other words, smooth optimization is a special case of composite optimization, where the nonsmooth term is identically zero. One particular example of composite optimization is the linear inverse problem with sparse representation, that is, the least square model (or linear regression) with $\ell_1$-regularization
\begin{equation}
\label{eqn: lasso}
\frac{1}{2} \|Ax - b \|^2 +\lambda \|x\|_1,\footnote{The rigorous definition of norms $\|\cdot\|$ and $\|\cdot\|_1$ are shown in~\Cref{subsec: notations-organization}.}
\end{equation}
where $A \in \mathbb{R}^{m\times d}$ is an $m\times d$ matrix,  $b \in \mathbb{R}^m$ is an $m$-dimensional vector and the regularization parameter $\lambda > 0$ is a tradeoff between fidelity to the measurements and noise sensitivity. Moreover, the composite objective function~\eqref{eqn: lasso} is widely applied in many fields, such as signal and image processing, statistical inference, geophysics,  astrophysics, and so on.~\Cref{fig: deblurring} shows an image deblurring problem with the original and observed (blurred and noisy) images of an elephant. Generally in the image deblurring problem, the matrix $A$ in the least square model with $\ell_1$-regularization~\eqref{eqn: lasso} is ill-conditioned~\citep{hansen2006deblurring}.
\begin{figure}[htpb!]
\centering
\begin{subfigure}[t]{0.48\linewidth}
\centering
\includegraphics[scale=0.5]{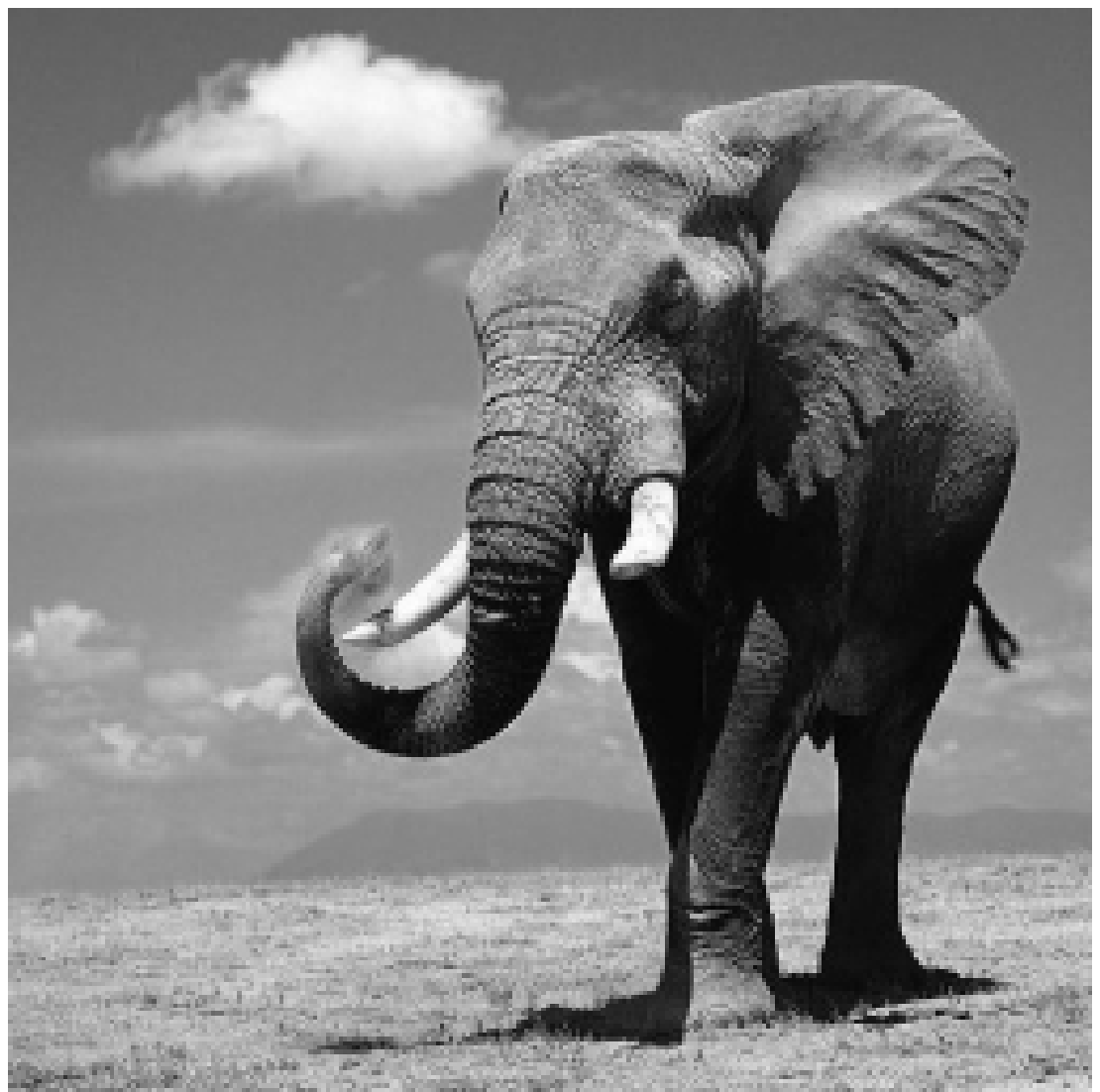}
\caption{Original}
\end{subfigure}
\begin{subfigure}[t]{0.48\linewidth}
\centering
\includegraphics[scale=0.5]{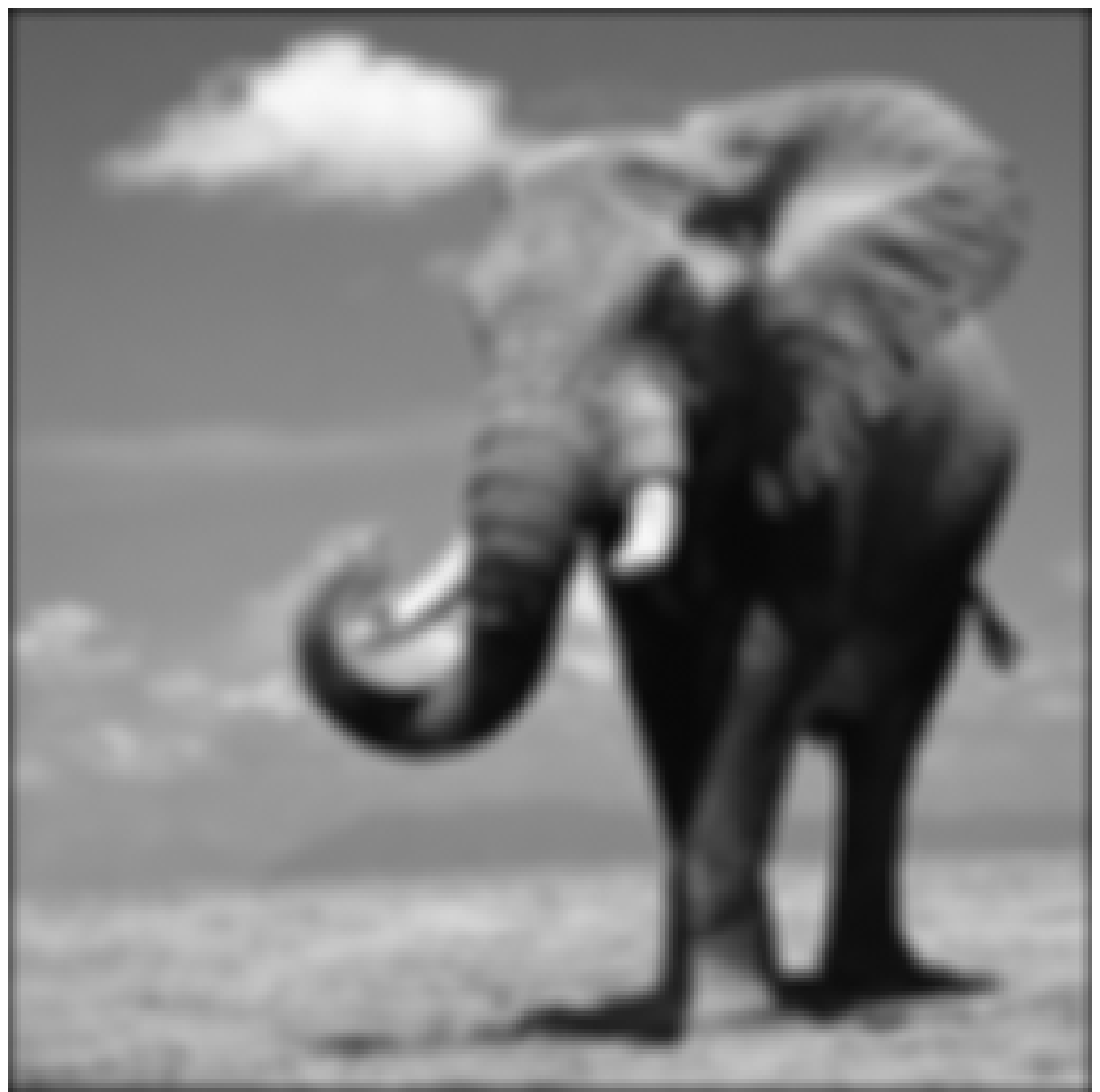}
\caption{Blurred and Noisy}
\end{subfigure}
\caption{The image deblurring problem: An image of an elephant}
\label{fig: deblurring}
\end{figure}

When the objective function is generalized from a smooth one to a composite one~\eqref{eqn: composite}, the vanilla gradient descent takes the general form by  introducing the $s$-proximal operator $P_s$\footnote{The $s$-proximal operator $P_s$ is defined rigorously by~\eqref{eqn: proximal-operator}  in~\Cref{subsec: notations-organization}.} as
\[
x_{k} = P_s\left(x_{k-1} - s\nabla f(y_{k-1})\right),
\]
which is named an iterative shrinkage-thresholding algorithm (ISTA).~\citet{beck2009fast} first provide a rigorous proof of ISTA's convergence rate, which is the same as the vanilla gradient descent~\eqref{eqn:1-over-k}. Furthermore, with the $s$-proximal operator $P_s$,~\texttt{NAG} can be generalized as
\[
\left\{\begin{aligned}
       & x_{k} = P_s\left(y_{k-1} - s\nabla f(y_{k-1})\right),              \\
       & y_{k} = x_{k}  + \frac{k-1}{k+r}(x_{k} - x_{k-1}),
       \end{aligned}\right.
\]
where the same convergence rate with~\texttt{NAG}~\eqref{eqn:1-over-k2} is also first proposed in~\citep{beck2009fast}. Hence, the $s$-proximal~\texttt{NAG} above is also called a fast iterative shrinkage-thresholding algorithm (FISTA). In~\Cref{fig: ista-fista}, we show the elephant's image in~\Cref{fig: deblurring} deblurred by ISTA and FISTA, where the parameters are set as: the regularization parameter $\lambda =10^{-6}$, the step size $s= 0.5$ and the iteration number $k=200$.

\begin{figure}[htpb!]
\centering
\begin{subfigure}[t]{0.48\linewidth}
\centering
\includegraphics[scale=0.5]{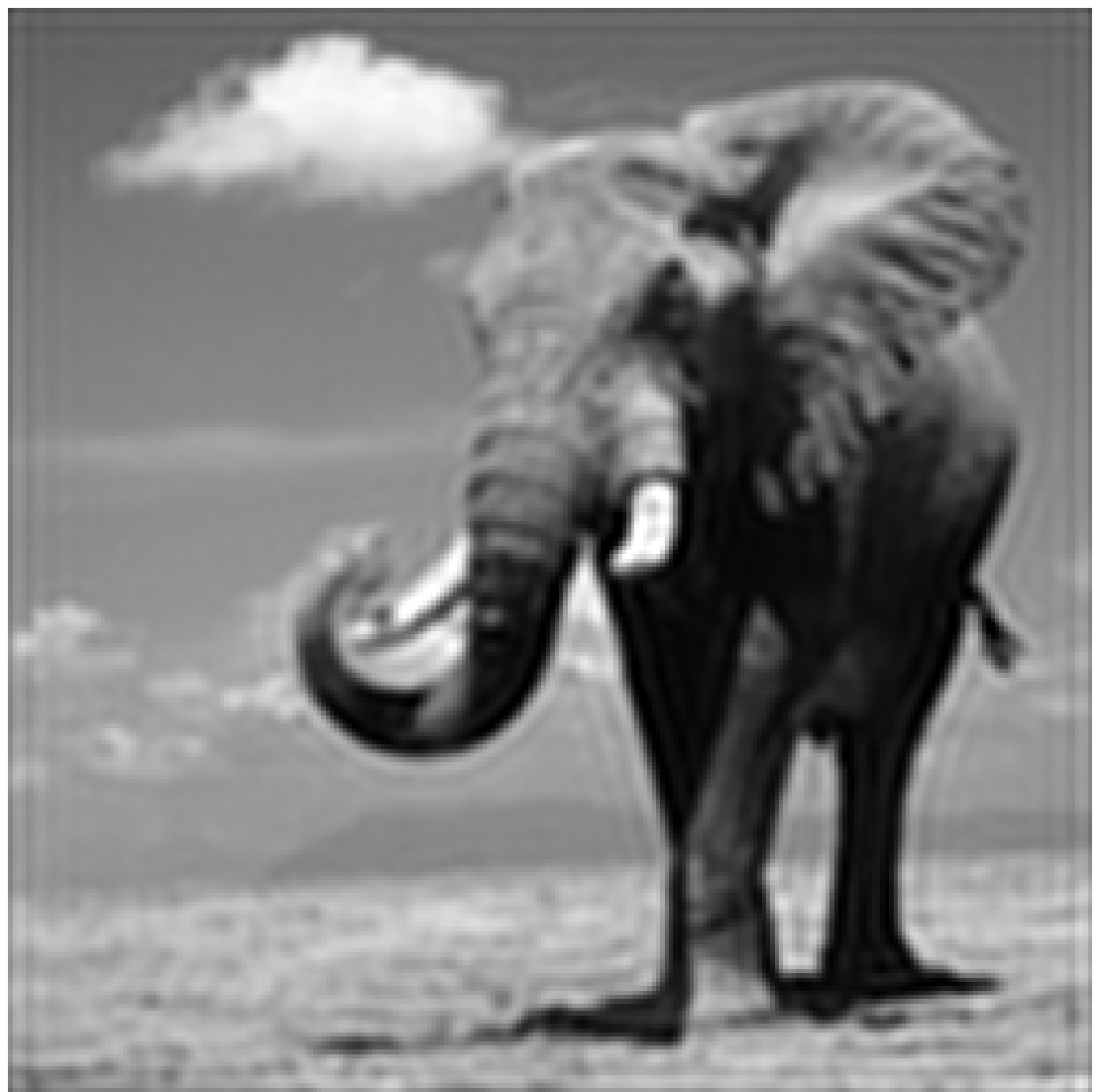}
\caption{ISTA}
\end{subfigure}
\begin{subfigure}[t]{0.48\linewidth}
\centering
\includegraphics[scale=0.5]{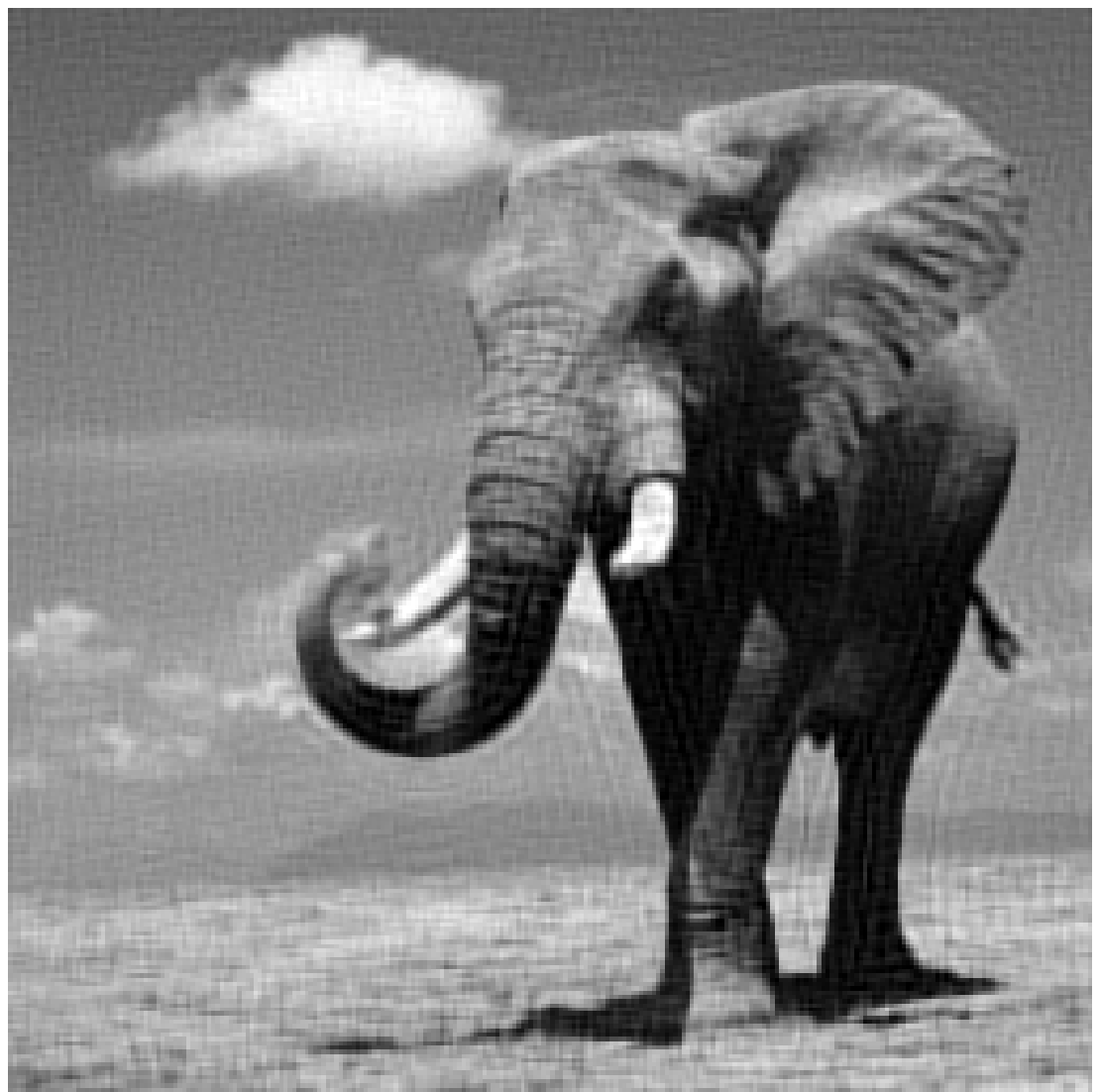}
\caption{FISTA}
\end{subfigure}
\caption{Deblurring of the elephant's image in~\Cref{fig: deblurring} by ISTA and FISTA.} 
\label{fig: ista-fista}
\end{figure}

Recently,~\citet{su2016differential} introduces the concept of $s$-proximal subgradient at $x \in \mathbb{R}^d$
\begin{equation}
\label{eqn: proximal-subgradient}
G_s(x) := \frac{x - P_s(x)}{s},
\end{equation}
which reduces ISTA and FISTA respectively to the forms of the vanilla gradient descent and the~\texttt{NAG} with the $s$-proximal subgradient $G_{s}(x_k)$ instead of the standard gradient $\nabla f(x_k)$. Moreover,~\citet{su2016differential} first uses the emerging Lyapunov analysis to obtain the convergence rate~\eqref{eqn:1-over-k2}, which is identified with that obtained in~\citep{beck2009fast}.

\subsection{Key observation}
\label{subsec: key-observation}

Recall the proofs for composite optimization in~\citep{beck2009fast, su2016differential}, the core ingredient is an inequality of the difference between $\Phi(y - sG_s(y))$ and $\Phi(x)$, shown in~\citep[Lemma 2.3]{beck2009fast} and~\citep[(22)]{su2016differential} respectively. Here in order to provide a clear picture for readers, we use the smooth objective function to demonstrate our key observation of how to improve the inequality of the difference between $\Phi(y - s\nabla f(y))$ and $\Phi(x)$, where $G_s(y)$ degenerates to $\nabla f(y)$ for smooth optimization. And then, the rigorous proof for the composite objective function is shown in~\Cref{sec: key-inequality}.

Given a smooth objective function $f$, then we separate the difference between $f(y - s\nabla f(y))$ and $f(x)$ into two parts as
\[
f(y - s\nabla f(y)) - f(x) =  \underbrace{f(y)- f(x)}_{\mathbf{I}}+  \underbrace{f(y - s\nabla f(y)) - f(y)}_{\mathbf{II}}. 
\]
\begin{itemize}
\item[(i)] For~\textbf{Part-I}, the conventional inequality used to estimate~\citep{beck2009fast, su2016differential} is the general convex inequality as
      \[
      \mathbf{I} = f(y)- f(x) \leq \left\langle \nabla f(y), y - x \right\rangle.  
      \] 
      By considering the condition of $L$-smoothness\footnote{The condition of $L$-smoothness is defined in~\Cref{subsec: notations-organization}.} further in~\citep{shi2021understanding, chen2022gradient},  the gradient norm minimization obtained is to improve the inequality of \textbf{Part-I} tighter as        
      \[
      \mathbf{I}= f(y)- f(x) \leq \left\langle \nabla f(y), y - x \right\rangle - \frac{1}{2L} \left\|\nabla f(y) - \nabla f(x) \right\|^2. 
      \] 
      However, the tighter inequality is fully dependent on the $L$-smooth condition, which cannot be generalized to the composite optimization.

\item[(ii)] Hence, we have to consider the inequality estimate of~\textbf{Part-II}. With the conditions of convexity and $L$-smoothness, we can estimate~\textbf{Part-II} as
\[
\mathbf{II} = f(y - s\nabla f(y)) - f(y) \leq - \left(s - \frac{Ls^2}{2}\right) \| \nabla f(y) \|^2.
\]
In~\citep[Lemma 2.3]{beck2009fast}, the step size is set as $s = 1/L$ and then $s- Ls^2/2 = 1/(2L)$. In~\citep[(22)]{su2016differential}, the Lipschitz constant is set as $L = 1/s$ and then $s- Ls^2/2 = s/2$. Actually, the step size $s$ is not necessarily set to be a constant $1/L$ but belongs to a range, that is, $s \in (0,1/L]$. Based on the key observation, we will provide a new proof of gradient norm minimization based on the Lyapunov function~\citep{su2016differential} and generalize it to the composite optimization. 
\end{itemize}

Combining~\textbf{Part-I} and~\textbf{Part-II}, we obtain the estimate of the difference between between $f(y - s\nabla f(y))$ and $f(x)$ as
\[
f(y - s\nabla f(y)) - f(x) \le \left\langle\nabla f(y), y - x \right\rangle  - \left(s - \frac{Ls^2}{2}\right) \| \nabla f(y) \|^2.
\] 
The generalization to the composite objective function is shown in~\Cref{sec: key-inequality}.

\subsection{Overview of contributions}
\label{subsec: overview-contribution}

In this paper, we study further the convergence rate of ISTA and FISTA using the techniques, phase-space representation and Lyapunov function, based 
on the high-resolution differential equation framework~\citep{shi2021understanding, chen2022gradient} and make contributions majorly about the subgradient norm minimization. Previously, only the convergence rate of objective values is shown in~\citep{beck2009fast, su2016differential}.


\paragraph{Proximal subgradient norm minimization of ISTA}
Recall the discrete Lyapunov function constructed in~\citep[Theorem A.6]{shi2019acceleration} is used to obtain the convergence rate in the gradient norm square. Based on the key observation of a pivotal inequality, we obtain the proximal subgradient norm minimization of ISTA converges as
\[
\min_{0\leq i \leq k}\|G_s(x_k)\|^2    \leq      o\left(\frac{1}{k^2}\right)  
\] 
for any $0 < s \leq 1/L$. 

\paragraph{Proximal subgradient norm minimization of FISTA}
Based on the key observation of a pivotal inequality, we take the phase-space representation technique from the high-resolution differential equation framework~\citep{shi2021understanding} into the discrete Lyapunov function constructed in~\citep[(22)]{su2016differential}. And then, we dig out the convergence rate of FISTA's squared proximal subgradient norm as

\[
\min_{0\leq i \leq k}\|G_s(y_k)\|^2    \leq      o\left(\frac{1}{k^3}\right)  
\] 
for any $0 < s <1/L$, by the gradient-correction scheme and the implicit-velocity scheme~\citep{chen2022gradient}, respectively.  Similarly, we also point out that the convergence rate of the objective values proposed in~\citep{su2016differential} can be improved to a faster rate as
\[
\min_{0\leq i \leq k} \Phi(x_i) - \Phi(x^\star) \leq  o\left(\frac{1}{k^2}\right)
\]
for any $0 < s \leq 1/L$. This is also available in practice and far more straightforward than the derivation in~\citep{attouch2016rate}.

\paragraph{Advantages of proximal subgradient norm minimization} Besides the merit that the proximal subgradient norm square converges faster than the objective values, there are still two advantages of the proximal subgradient norm beyond the objective values in practice. The first one is that 
the global minimum $\Phi(x^\star)$ is practically unknown so that the convergence rate with the ordinate ($y$-axis) as the logarithmic scale is always observed by $\log \Phi(x_i)$ instead of $\log(\Phi(x_i) - \Phi(x^\star) )$. This phenomenon leads to the iterative curve converging to the constant $\log \Phi(x^\star)$ other than $-\infty$, which fails to characterize the property of convergence, e.g., linear or sublinear convergence. However, the proximal subgradient norm $\|G_s(x_k)\|$ or $\|G_s(y_k)\|$ always converges to zero, where the log-scale ordinate problem above does not exist.  Moreover, the second one is  the stop criteria, which is generally adopted for the objective values by $\Phi(x_{k+1})- \Phi(x_{k}) < \epsilon$, where the parameter $\epsilon > 0$ is a small number. Moreover, the convergence rate of  $\Phi(x_{k+1})- \Phi(x_{k})$ in practice is hard to be estimated, so we can only empirically try to adjust the parameter $\epsilon$. However, for the proximal subgradient norm, we can directly estimate to take about $O(1/\epsilon^{\frac23})$ iterative times  into the stop criteria $\|G_s(x_k)\| < \epsilon$ or $\|G_s(y_k)\| < \epsilon$.  
%
%

\subsection{Notations and Organization}
\label{subsec: notations-organization}

In this paper, we follow the notions of~\citep{nesterov2003introductory} with slight modifications. Let $\mathcal{F}^0(\mathbb{R}^d)$ be the class of continuous convex functions defined on $\mathbb{R}^d$; that is, $g \in \mathcal{F}^0$ if $g(y) \geq g(x) + \langle \nabla g(x), y - x\rangle$ for any $x, y \in \mathbb{R}^d$. The function class $\mathcal{F}^1_{L}(\mathbb{R}^d)$ is the subclass of $\mathcal{F}^0(\mathbb{R}^d)$ with the $L$-smooth gradient; that is, if $f \in \mathcal{F}^1_{L}(\mathbb{R}^d)$, its gradient is $L$-smooth in the sense that 
\[
\| \nabla f(x) - \nabla f(y) \| \leq L \|x - y\|,
\]
where $\|\cdot\|$ denotes the standard Euclidean norm and $L > 0$ is the Lipschitz constant. (Note that this implies that $\nabla f$ is also $L'$-Lipschitz for any $L' \geq L$.) For any $g \in \mathcal{F}^0(\mathbb{R}^d)$, we call the vector $v \in \mathbb{R}^d$ a subgradient of $g$ at $x \in \mathbb{R}^d$ if it satisfies $g(y) \geq g(x) + v^T(y - x)$ for any $y \in \mathbb{R}^d$. And the set of all subgradients of $g \in \mathcal{F}^0(\mathbb{R}^d)$ at $x\in \mathbb{R}^d$ is called the subdifferential of $g$ at $x\in \mathbb{R}^d$ as
\[
\partial g(x) := \left\{ v \in \mathbb{R}^d: g(y) \geq g(x) + v^T(y - x),\;\forall y \in \mathbb{R}^d \right\}.
\]
Finally, we give the definition of the $s$-proximal operator~\citep{beck2009fast, su2016differential}. For any $f\in \mathcal{F}_L^1$ and $g \in \mathcal{F}^0$, the $s$-proximal operator is defined as
\begin{equation}
\label{eqn: proximal-operator}
    P_s(x) := \mathop{\arg\min}_{z\in\mr^d}\left\{ \frac{1}{2s}\left\| z - \left(x - s\nabla f(x)\right) \right\|^2 + g(z) \right\},
\end{equation}
for any $x \in \mathbb{R}^d$. For the special case, when $g$ is the $\ell_1$-norm, that is, $g(x) = \lambda \|x\|_1$,\footnote{For any vector $x \in \mathbb{R}^d$, its $\ell_1$-norm $\|\cdot\|_1$ is defined as \[\|x\|_1 = \sum_{i=1}^{d}|x_i|.\]} we can obtain the closed-form expression of the $s$-proximal operator~\eqref{eqn: proximal-operator} as
\[
P_s(x)_i = \big(\left|\left(x - s\nabla f(x)\right)_i\right| - \lambda s\big)_+ \text{sgn}\big(\left(x - s\nabla f(x)\right)_i\big),
\]
where $i=1,\ldots,d$ is the index.

The remainder of this paper is organized as follows. In~\Cref{sec: related}, some related research works are described. We rigorously prove a pivotal inequality for composite functions in~\Cref{sec: key-inequality}. The proximal subgradient minimization of ISTA and FISTA are shown in~\Cref{sec: nonaccele} and~\Cref{sec: accele}, respectively. Moreover, we demonstrate the proximal subgradient minimization of FISTA by both the gradient-correction scheme and the implicit-velocity scheme in~\Cref{sec: accele}. In~\Cref{sec: discussion}, we propose a conclusion and discuss some further research works.

%
%
%
%
%
%
%
%
%
%
%

\section{Related works}
\label{sec: related}

Since the acceleration phenomenon was discovered by~\citet{nesterov1983method}, it has been an important topic to explore the cause leading to the acceleration in the first-order optimization. Until recently, the cause has not been successfully uncovered by the gradient correction term or its equivalent implicit-velocity form~\citep{shi2021understanding, chen2022gradient}. In other words, a theoretical guarantee is well-established for the convex smooth unconstrained optimization. A detailed review with many emerging perspectives is shown in~\citep{jordan2018dynamical}. However, there has not been a complete theoretical system for composite optimization widely used in practice, e.g., linear inverse problems with sparse representation in signal and image processing, statistical inference, geophysics, astrophysics, etc. The study on the convergence rate of composite objective functions starts from~\citep{beck2009fast} for the linear inverse problem with sparse representation.~\citet{su2016differential} introduce the technique of Lyapunov analysis and simplify the proof in~\citep{beck2009fast}. 

The least-square model with $\ell_1$-regularization is used widely in the linear inverse problem and is a classical statistical model performing the sparse variable selection in matching pursuit, which is named least absolute shrinkage and selection operator, shortened as Lasso~\citep{tibshirani1996regression}. Afterward, the $\ell_1$-norm also is used in basis pursuit~\citep{chen2001atomic}. Perhaps, one of the significant applications for $\ell_1$-norm related to optimizations is compressed sensing, which is a signal processing technique to obtain solutions to underdetermined linear systems for efficiently acquiring and reconstructing a signal. The famous Nyquist-Shannon theory~\citep{nyquist1928certain, shannon1949communication} says that the $\ell_1$-norm optimization can exploit the sparsity of a signal to recover it using far fewer samples than required. 
Recently, the theory about compressed sensing has been developed further in~\citep{donoho2006compressed, donoho2006most,candes2006stable, candes2006robust}. More specially, when the knowledge of a signal's sparsity is known, the number of samples can be reduced further. Meanwhile, the characteristics of $\ell_1$-norm for the sparsity are generalized to its matrix version --- nuclear norm, which is used for matrix completion, e.g., the Netflix problem~\citep{candes2012exact}.

Another venerable line of works combining the low-resolution differential equation and the continuous limit of the Newton method to design and analyze new algorithms is proposed in~\citep{alvarez2000minimizing, attouch2012second, attouch2014dynamical,  attouch2016fast, attouch2016rate, attouch2020first, attouch2022fast, attouch2022ravine}, of which the terminology
is called inertial dynamics with a Hessian-driven term. Although the inertial dynamics with a
Hessian-driven term resembles closely with the high-resolution differential equations in~\citep{shi2021understanding}, it is important to note that the Hessian-driven terms are from the second-order information of Newton’s method~\citep{attouch2014dynamical}, while the gradient correction entirely relies on the first-order information of Nesterov's accelerated gradient descent method itself. 

\section{A pivotal inequality for composite optimization}
\label{sec: key-inequality}

In this section, we provide a refined inequality for~\citep[Lemma 2.3]{beck2009fast} and~\citep[(22)]{su2016differential}, based on a meticulous observation of the $s$-proximal operator $P_s(\cdot)$. Here, we use the step size $s \in (0,1/L)$ instead of $1/L$ to emphasize the subscript of $P_s$,  which is pivotal for the composite objective function to bring about the proximal subgradient norm minimization and its acceleration.

\begin{lemma}
\label{lem: key-inequality}
Let $\Phi = f + g$ be a composite function with $f \in \mathcal{F}_{L}^1(\mathbb{R}^d)$ and $g \in \mathcal{F}^0(\mathbb{R}^d)$. Then, the composite function $\Phi$ satisfies the following inequality 
    \begin{equation}
    \label{eqn: key-inequality}
        \Phi(y - sG_s(y)) \le \Phi(x) + \langle G_s(y), y - x\rangle - \left(s - \frac{s^2 L}{2}\right)\|G_s(y)\|^2.
    \end{equation}
\end{lemma}

\begin{proof}
Similar to the proof in~\citep[Lemma 2.3]{beck2009fast}, we first define the quadratic approximation of $\Phi(x)=f(x)+g(x)$ at a given point $y \in \mathbb{R}^d$ for any step size $0 < s \leq 1/L$: 
\begin{equation}
\label{eqn: quadratic-approx}  
Q_s(x, y) := f(y) + \langle\nabla f(y),  x - y\rangle + \frac{1}{2s}\|x - y\|^2 + g(x).
\end{equation}
Some basic calculations tell us that the proximal value at the given point $y$ is the unique minimizer of the quadratic approximation~\eqref{eqn: quadratic-approx} as
\[
P_s(y) = \mathop{\argmin}_{x\in \mathbb{R}^d}Q_s(x, y),
\]
that is, there exists a subgradient $\gamma(y) \in \partial g(P_s(y))$ such that the optimal condition is satisfied as
\begin{equation}
\label{eqn: prox-optimality}
\nabla f(y) + \frac{1}{s} \big( P_s(y) - y \big) + \gamma(y) = 0.
\end{equation}

Furthermore, since the convex function $f$ is $L$-smooth, it follows that
\[
\Phi(x) = f(x) + g(x) \leq   f(y) + \langle \nabla f(y), x - y\rangle + \frac{L}{2}\|x - y\|^2 + g(x)\leq Q_{\frac1L}(x,y), 
\] 
for any $x,y \in \mathbb{R}^d$. Substituting $x$ by the proximal value $P_s(y)$, we can obtain
\begin{equation}
\label{eqn: prox-inequality}
\Phi(x) - \Phi(P_s(y)) \ge \Phi(x) - Q_{\frac{1}{L}}(P_s(y), y).
\end{equation}

Finally, by the convexity of both $f(x)$ and $g(x)$,  the following two inequalities 
    \begin{subequations} 
    \begin{align}
        & f(x) \ge f(y) + \langle\nabla f(y), x- y\rangle,               \label{eqn: smooth-convex}    \\
        & g(x) \ge g(P_s(y)) + \langle \gamma(y), x - P_s(y)\rangle,     \label{eqn: nonsmooth-convex}  
    \end{align}
hold for any $x, y \in \mathbb{R}^d$. 
   \end{subequations}
Summing the two inequalities above,~\eqref{eqn: smooth-convex} and~\eqref{eqn: nonsmooth-convex}, yields 
    \begin{equation}
    \label{eqn: convex-inequality}  
    \Phi(x) = f(x) + g(x) \ge f(y) + \langle \nabla f(y), x - y\rangle + g(P_s(y)) + \langle \gamma(y), x - P_s(y)\rangle.
    \end{equation}
From the two inequalities~\eqref{eqn: prox-inequality} and~\eqref{eqn: convex-inequality}, we have
\begin{equation}
\label{eqn: final-key-1}
\Phi(x) - \Phi(P_s(y))\ge f(y) + \langle \nabla f(y), x - y\rangle + g(P_s(y)) + \langle \gamma(y), x - P_s(y)\rangle - Q_{\frac{1}{L}}(P_s(y), y).
\end{equation}
With the definition of the quadratic approximation~\eqref{eqn: quadratic-approx}, inequality~\eqref{eqn: final-key-1} can be calculated as
\begin{equation}
\label{eqn: final-key-2}
\Phi(x) - \Phi(P_s(y))\ge  \langle \nabla f(y) + \gamma(y), x - P_s(y)\rangle - \frac{L}{2} \left\| P_s(y) - y\right\|^2.
\end{equation}
Putting the optimal condition~\eqref{eqn: prox-optimality} into~\eqref{eqn: final-key-2}, we have
\begin{align*}
\Phi(x) - \Phi(P_s(y)) & \ge  \frac{1}{s}\langle y - P_s(y), x - P_s(y)\rangle - \frac{L}{2} \left\| P_s(y) - y\right\|^2 \\
                       & =  \left\langle G_s(y), x - y\right\rangle  + \left( s - \frac{Ls^2}{2}\right) \left\| G_s(y)\right\|^2,
\end{align*}
where the last equality follows from the definition of the $s$-proximal subgradient~\eqref{eqn: proximal-subgradient} directly. Hence, the proof is complete.

\end{proof}

%
%
%
%

\section{Proximal subgradient norm minimization of ISTA}
\label{sec: nonaccele}

In this section, we investigate the convergence rate of ISTA for composite optimization in both the objective value and the proximal subgradient norm square. With the definition of the $s$-proximal subgradient~\eqref{eqn: proximal-subgradient}, ISTA can be rewritten as
\begin{equation}
\label{eqn: ista-gs}
x_{k} = x_{k-1} - sG_s(x_{k-1}). 
\end{equation}
From~\citep[Theorem A.6]{shi2019acceleration}, the discrete Lyapunov function for the vanilla gradient descent is constructed as
\[
    \E(k) = sk(f(x_{k}) - f(x^\star)) + \frac{1}{2}\|x_k - x^\star\|^2.
\]
Here, for the discrete Lyapunov function of ISTA, we use the composite function $\Phi(x)$ instead of the smooth function $f(x)$  as
\begin{equation}
\label{eqn: lyapunov-ista}
    \E(k) = sk(\Phi(x_{k}) - \Phi(x^\star)) + \frac{1}{2}\|x_k - x^\star\|^2.
\end{equation} 
With the discrete Lyapunov function~\eqref{eqn: lyapunov-ista}, we show the convergence rate of both the objective value and the gradient norm square in the following theorem.

\begin{theorem}
\label{thm: ista}
Let $\Phi = f + g$ be a composite function with $f \in \mathcal{F}_{L}^1(\mathbb{R}^d)$ and $g \in \mathcal{F}^0(\mathbb{R}^d)$. Taking any step size $0 < s \leq 1/L$, the iterative sequence $\{x_k\}_{k=0}^{\infty}$ generated by ISTA satisfies
    \begin{equation}
    \label{eqn: ista-s}
 \Phi(x_k) - \Phi(x^\star) \leq \frac{\|x_0 - x^\star\|^2}{2sk}, \quad \min_{0\leq i \leq k} \| G_s(x_i) \|^2 \leq \frac{2\|x_0 - x^\star\|^2}{3s^2 k(k+1)};   
    \end{equation}
Furthermore, the iterative sequence $\{x_k\}_{k=0}^{\infty}$ also satisfies
\begin{equation}
\label{eqn: ista-s1}
\lim_{k\rightarrow \infty} \left[k(k+1)\min_{0\leq i \leq k} \| G_s(x_i) \|^2\right] = 0.
\end{equation}
\end{theorem}

\begin{proof}
With the discrete Lyapunov function~\eqref{eqn: lyapunov-ista}, we calculate the iterative difference between $\mathcal{E}(k+1)$ and $\mathcal{E}(k)$ as 
\begin{align}
 \E&(k+1) - \E(k) \nonumber\\
        &= sk(\Phi(x_{k+1}) - \Phi(x_k)) + s(\Phi(x_{k+1}) - \Phi(x^\star)) + \frac{1}{2}\inner{x_{k+1} - x_k, x_{k+1} + x_k - 2x^\star} \nonumber\\
        & = sk(\Phi(x_{k+1}) - \Phi(x_k)) + s(\Phi(x_{k+1}) - \Phi(x^\star)) - s\inner{G_s(x_k), x_k - x^\star} + \frac{s^2}{2}\norm{G_s(x_k)}^2. \label{eqn: diff-lyapunov}
    \end{align}
With~\Cref{lem: key-inequality},  we obtain the following two inequalities by putting the equivalent form of ISTA~\eqref{eqn: ista-gs} into the pivotal inequality~\eqref{eqn: key-inequality}, as
\begin{subequations}
\begin{align}
& \Phi(x_{k+1}) - \Phi(x_k) \le - \left(2s - \frac{s^2 L}{2}\right)\norm{G_s(x_k)}^2,\label{eqn: convex-inq-1}\\
& \Phi(x_{k+1}) - \Phi(x^\star) \le \inner{G_s(x_k), x_k - x^\star} - \left(s - \frac{s^2 L}{2}\right)\norm{G_s(x_k)}^2 \label{eqn: convex-inq-2}.
\end{align}
\end{subequations}
Substituting the two inequalites~\eqref{eqn: convex-inq-1} and~\eqref{eqn: convex-inq-2} into~\eqref{eqn: diff-lyapunov}, we obtain the iterative difference between $\mathcal{E}(k+1)$ and $\mathcal{E}(k)$ as 
    \begin{align*}
        \E(k+1) - \E(k) &\le - \frac{s^2}{2}\left[ 2(2k+1) - (k+1)sL\right]\norm{G_s(x_k)}^2 + \frac{s^2}{2}\norm{G_s(x_k)}^2\\
        & \le -\frac{3ks^2}{2}\norm{G_s(x_k)}^2,
    \end{align*}
where the last inequality follows from $sL \leq 1$. Hence, we complete the proof with some basic calculations.
\end{proof}

We can find that the step size $s=1/L$ is available in the proof of the ISTA's proximal subgradient norm minimization above. In other words, we can also obtain the ISTA's proximal subgradient norm minimization above by the inequality in~\citep[Lemma 2.3]{beck2009fast} and~\citep[(22)]{su2016differential}. Hence, we conclude this section with the following corollary by setting the step size $s = 1/L$.
\begin{corollary}
\label{coro: ista}
When the step size is set $s = 1/L$,  the iterative sequence $\{x_k\}_{k=0}^{\infty}$ generated by ISTA satisfies
    \begin{equation}
    \label{eqn: ista-l}
 f(x_k) - f(x^\star) \leq \frac{L\|x_0 - x^\star\|^2}{2k}, \quad \min_{0\leq i \leq k} \| G_s(x_i) \|^2 \leq \frac{2L^2\|x_0 - x^\star\|^2}{3 k(k+1)}.   
    \end{equation}
\end{corollary}
\section{Proximal subgradient norm minimization of FISTA}
\label{sec: accele}

In this section, we investigate the convergence rate of FISTA for composite optimization in both the objective value and the gradient norm square. With the definition of the $s$-proximal subgradient~\eqref{eqn: proximal-subgradient}, FISTA can be rewritten as
\begin{equation}
\label{eqn: fista-gs}
\left\{ 
\begin{aligned}
    &x_{k} = y_{k-1} - sG_s(y_{k-1}),\\
    &y_{k} = x_{k} + \frac{k-1}{k+r}(x_{k} - x_{k-1}),
\end{aligned}
\right.
\end{equation}
with $r\ge 2$. Here, we provide a discrete Lyapunov function based on~\citep{su2016differential} with a slight modification as 
\begin{equation}
\label{eqn: lyapunov-fista}
 \mathcal{E}(k) = sk(k+r)(\Phi(x_k) - \Phi(x^\star)) + \frac{1}{2}\|\sqrt{s}(y_{k}-x_{k}) + r(y_k - y^\star)\|^2.
\end{equation} 
With the discrete Lyapunov function above~\eqref{eqn: lyapunov-fista}, we show the convergence rate of both the objective value and the gradient norm square in the following theorem.

\begin{theorem}
\label{thm: fista}
Let $\Phi = f + g$ be a composite function with $f \in \mathcal{F}_{L}^1(\mathbb{R}^d)$ and $g \in \mathcal{F}^0(\mathbb{R}^d)$. Taking any step size $0 < s \leq 1/L$, the iterative sequence $\{x_k\}_{k=0}^{\infty}$ generated by FISTA satisfies
    \begin{equation}
    \label{eqn: fista-obj}
 \Phi(x_k) - \Phi(x^\star) \leq \frac{r^2\|x_0 - x^\star\|^2}{2s(k+1)(k+r+1)};   
    \end{equation}
if the step size satisfies $0< s< 1/L$, the iterative sequence $\{y_k\}_{k=0}^{\infty}$ satisfies
\begin{equation}
\label{eqn: fista-grad}
\left\{\begin{aligned}
& \min_{0\leq i \leq k} \| G_s(y_i) \|^2 \leq \frac{6r^2\|x_0 - x^\star\|^2}{s^2(1-sL) (k+1)\left(2k^2 + (6r+1)k + 6r^2 \right)},\\
& \lim_{k\rightarrow \infty} \left(k^3\min_{0\leq i \leq k} \| G_s(y_i) \|^2\right) = 0.
\end{aligned} \right.
\end{equation}
Furthermore, if $r>2$, the iterative sequence $\{x_k\}_{k=0}^{\infty}$ also satisfies
\begin{equation}
\label{eqn: fista-obj-fast}
\lim_{k\rightarrow \infty} \left[k^2 \left(\min_{0\leq i \leq k} \Phi(x_k) - \Phi(x^\star)\right)\right] = 0.
\end{equation}
\end{theorem}

Here, we provide two proofs for~\Cref{thm: fista} by both the gradient-correction scheme and the implicit-velocity scheme~\citep{chen2022gradient} in the following sections, respectively.

\subsection{The gradient-correction scheme}
\label{subsec: gradient-correction-scheme}
Similarly to~\citep{chen2022gradient}, with the velocity iterates $v_{k-1}= (y_k - y_{k-1})/\sqrt{s}$, we obtain the gradient-correction phase-space representation of FISTA~\eqref{eqn: fista-gs} as  
\begin{equation}
\label{eqn: grad-correct-phasespace}
    \left\{\begin{aligned}
        & y_k - y_{k-1} = \sqrt{s}v_{k-1},\\
        & v_k - v_{k-1} = -\frac{r+1}{k}v_k - \left(1 + \frac{r+1}{k}\right)\sqrt{s}G_s(y_k) - \sqrt{s}(G_s(y_k) - G_s(y_{k-1})),
            \end{aligned}\right.
\end{equation}
with any initial value $y_0\in\mathbb{R}^d$ and $v_0 = -\sqrt{s}G_s(y_0)$. It is noted here that the gradient-correction phase-space representation~\eqref{eqn: grad-correct-phasespace} is an explicit-implicit scheme. Correspondingly, the discrete Lyapunov function~\eqref{eqn: lyapunov-fista} is reformulated as
\begin{equation}
\label{eqn: lyapunov-fista-grad-correction}
 \mathcal{E}(k) = sk(k+r)(\Phi(x_k) - \Phi(x^\star)) + \frac{1}{2}\|\sqrt{s}kv_{k-1} + r(y_k - y^\star) + skG_s(y_{k-1})\|^2.
\end{equation} 
Furthermore, we reformulate the second iterative equality of the gradient-correction phase-space representation~\eqref{eqn: grad-correct-phasespace} as
\begin{equation}
\label{eqn:phase-ista2}
(k+r+1)v_{k} - kv_{k-1} + \sqrt{s}\left[(k+r+1)\nabla G_s(y_k) - k\nabla G_s(y_{k-1}) \right] =  - k\sqrt{s}G_s(y_k).
\end{equation}
for convenience.

\begin{proof}[Proof of~\Cref{thm: fista} in the gradient-correction scheme]
Here, we present the proof in three steps to make it clear.

\begin{itemize}
\item[\textbf{(1)}] In order to obtain the iterative difference of the Lyapunov function~\eqref{eqn: lyapunov-fista-grad-correction}, $\mathcal{E}(k+1) - \mathcal{E}(k)$, the key point here is to compute the iterative difference of the second term as
\begin{align*}
      & \left[\sqrt{s}(k+1)v_{k} + r(y_{k+1} - x^\star) + s(k+1)G_s(y_{k}) \right] - \left[\sqrt{s}kv_{k-1} + r(y_{k} - x^\star) + skG_s(y_{k-1}) \right] \\
  =   & \sqrt{s}(k+r+1)v_{k} - \sqrt{s}kv_{k-1} + s(k+1)G_s(y_{k}) - skG_s(y_{k-1}). 
\end{align*}
Combined with the second iterative inequality~\eqref{eqn:phase-ista2}, the iterative difference is calculated as
\begin{align}
      & \left[\sqrt{s}(k+1)v_{k} + r(y_{k+1} -x^\star) + s(k+1)G_s(y_{k}) \right] - \left[\sqrt{s}kv_{k-1} + r(y_{k} - x^\star) + skG_s(y_{k-1}) \right] \nonumber \\
  =   & - s(k+r)G_s(y_{k}). \label{eqn:difference-a}
\end{align}

\item[\textbf{(2)}] Then, with the subscripts labeled, we calculate the iterative difference of the Lyapunov function~\eqref{eqn: lyapunov-fista-grad-correction} as
\begin{align*}
\mathcal{E}&(k+1) - \mathcal{E}(k) \\
& =  \underbrace{ sk(k+r)\left( \Phi(x_{k+1}) - \Phi(x_{k}) \right) + s(2k+r+1)\left( \Phi(x_{k+1}) - \Phi(x^\star) \right) }_{\mathbf{I}}  + \frac{s^2(k+r)^2}{2}\big\| G_s(y_{k}) \big\|^2\\
                                  & \mathrel{\phantom{=}} - \underbrace{s(k+r)\left\langle G_s(y_{k}), \sqrt{s}kv_{k-1} + r(y_{k} - x^\star) + skG_s(y_{k-1})\right\rangle}_{\mathbf{II}} \\
& =   \underbrace{sk(k+r)\left( \Phi(x_{k+1}) - \Phi(x_{k}) \right) }_{\mathbf{I}_1}+ \underbrace{s(2k+r+1)\left( \Phi(x_{k+1}) - f(x^\star) \right)}_{\mathbf{I}_2}  + \frac{s^2(k+r)^2}{2}\big\| G_s(y_{k}) \big\|^2\\
                                  & \mathrel{\phantom{=}} - \underbrace{sk(k+r)\left\langle G_s(y_{k}), y_{k} - y_{k-1} \right\rangle }_{\mathbf{II}_1}- \underbrace{sr(k+r)\left\langle G_s(y_{k}),  y_{k} - x^\star\right\rangle}_{\mathbf{II}_2} - \underbrace{s^{2}k(k+r)\langle G_s(y_{k}), G_s(y_{k-1}) \rangle}_{\mathbf{II}_3} 
	\end{align*}

\item[\textbf{(3)}] Here, this step is the essential difference with~\citep[Section 3]{chen2022gradient}. With~\Cref{lem: key-inequality}, we can obtain the following two inequalites by putting $x_{k+1}$, $x_k$ and $x_{k+1}$, $x^\star$ into~\eqref{eqn: key-inequality} as  
    \begin{align}
        &\Phi(x_{k+1}) - \Phi(x_k) \le \langle G_s(y_k), y_k - x_k\rangle - \left(s - \frac{s^2 L}{2}\right)\|G_s(y_k)\|^2, \label{eqn: key-1-inq}\\
        &\Phi(x_{k+1}) - \Phi(x^\star) \le \langle G_s(y_k), y_k - x^\star\rangle - \left(s - \frac{s^2 L}{2}\right)\|G_s(y_k)\|^2. \label{eqn: key-2-inq}
    \end{align}
With the inequality~\eqref{eqn: key-1-inq}, we obtain the estimate for $\mathbf{I}_1 - \mathbf{II}_1 -  \mathbf{II}_3$ as
\begin{align*}
\mathbf{I}_1 &- \mathbf{II}_1 -  \mathbf{II}_3 \\
         & =        sk(k+r)\left( \Phi(x_{k+1}) - \Phi(x_{k}) \right) - sk(k+r)\left\langle G_s(y_{k}), y_{k} -  y_{k-1} \right\rangle - s^{2}k(k+r)\langle G_s(y_{k}), G_s(y_{k-1}) \rangle\\
                             & \leq   - sk(k+r)\left\langle G_s(y_{k}), x_{k} -  y_{k-1} \right\rangle  - k(k+r)\left(s^2 - \frac{s^3 L}{2}\right)\|G_s(y_k)\|^2 - s^{2}k(k+r)\langle G_s(y_{k}), G_s(y_{k-1}) \rangle \\
                             & =  - k(k+r)\left(s^2 - \frac{s^3 L}{2}\right)\|G_s(y_k)\|^2.    
\end{align*}
Moreover, with the inequality~\eqref{eqn: key-2-inq},  we calculate the difference between $\mathbf{I}_2$ and $\mathbf{II}_2$ as
\begin{align*}
\mathbf{I}_2 - \mathbf{II}_2 & =   s(2k+r+1)\left( \Phi(x_{k+1}) - f(x^\star) \right) - sr(k+r)\left\langle G_s(y_{k}),  y_{k} - x^\star\right\rangle \\
                             & \leq  -  r(k+r)\left(s^2 - \frac{s^3 L}{2}\right)\|G_s(y_k)\|^2- s \left[ (r-2)k + r^2 - r - 1  \right]\left( \Phi(x_{k+1}) - \Phi(x^\star) \right).
\end{align*}

Hence, the iterative difference of the Lyapunov function~\eqref{eqn: lyapunov-fista-grad-correction} is computed as
\begin{align}
\mathcal{E}(k+1) - \mathcal{E}(k) \leq & - \frac{s^2(k+r)^2}{2}\left(1- sL\right)\|G_s(y_k)\|^2 \nonumber \\
                                       & - s \left[ (r-2)k + r^2 - r - 1  \right]\left( \Phi(x_{k+1}) - \Phi(x^\star) \right). \label{eqn: final-iterative}   
\end{align}
Obviously, the convergence rate of the objective values~\eqref{eqn: fista-obj} is obtained. When the step size satisfies $0<s<1/L$, we can derive the gradient norm minimization~\eqref{eqn: lyapunov-fista-grad-correction} by summarizing the ineuqality~\eqref{eqn: final-iterative} from $0$ to $k$ and from $\left \lfloor  k/2 \right \rfloor$ to $k$ and then taking the minimum of the proximal subgradient norm. Furthermore, summarizing the inequality from $\left \lfloor  k/2 \right \rfloor$ to $k$, we take the minimum of the objective value $\Phi(x_{k+1}) - \Phi(x^\star)$ to obtain the faster rate of objective values~\eqref{eqn: fista-obj-fast}.
\end{itemize}
\end{proof}

\subsection{The implicit-velocity scheme}
\label{subsec: implicit-velocity-scheme}

Similarly to~\citep{chen2022gradient}, with the velocity iterates $v_{k}= (x_k - x_{k-1})/\sqrt{s}$, we obtain the implicit-velocity phase-space representation of FISTA~\eqref{eqn: fista-gs} as  
\begin{equation}
\label{eqn: implicit-velocity-phasespace}
\left\{\begin{aligned}
& x_{k+1} - x_{k} = \sqrt{s}v_{k+1}, \\
&v_{k+1} - v_{k} = -\frac{r+1}{k+r}v_k - \sqrt{s}G_s\left(x_k + \frac{k-1}{k+r}\sqrt{s}v_k\right),
\end{aligned} \right.
\end{equation}
with any initial value $x_0\in\mathbb{R}^d$ and $v_0 = 0$. It is noted here that the implicit-velocity phase-space representation~\eqref{eqn: implicit-velocity-phasespace} is an implicit-explicit scheme. Correspondingly, the discrete Lyapunov function~\eqref{eqn: lyapunov-fista} is reformulated as
\begin{equation}
\label{eqn: lyapunov-fista-implicit-velocity}
\mathcal{E}(k) = sk(k+r)(\Phi(x_k) - \Phi(x^\star)) + \frac{1}{2}\|\sqrt{s}(k-1)v_{k} + r(x_k - x^\star)\|^2.
\end{equation}
Furthermore, we reformulate the second iterative equality of the implicit-velocity phase-space representation~\eqref{eqn: implicit-velocity-phasespace} as
\begin{equation}
\label{eqn: phase-xy-1}
    (k+r)v_{k+1} - (k-1)v_k = -(k+r)\sqrt{s}G_s\left(y_k\right)
\end{equation}
and the second equality of the original FISTA as
\begin{equation}
\label{eqn: nag-2-yx}
    y_k = x_k + \frac{k-1}{k+r}\cdot\sqrt{s}v_k
\end{equation}
for convenience.

\begin{proof}[Proof of~\Cref{thm: fista} in the implicit-velocity scheme]
Here, we also present the proof in three steps to make it clear.

\begin{itemize}
\item[\textbf{(1)}] In order to obtain the iterative difference of the Lyapunov function~\eqref{eqn: lyapunov-fista-implicit-velocity}, $\mathcal{E}(k+1) - \mathcal{E}(k)$, the key point here is to compute the iterative difference of the second term as
\begin{multline*}
 \left[\sqrt{s}kv_{k+1}+r(x_{k+1}-x^\star)\right]  - \left[\sqrt{s}(k-1)v_k+r(x_k-x^\star)\right] \\
    =  \sqrt{s}kv_{k+1} - \sqrt{s}(k-1)v_{k} + r(x_{k+1}-x_{k}). 
\end{multline*}
Combined with the second iterative inequality~\eqref{eqn: phase-xy-1}, the iterative difference is calculated as
\begin{align}
 \left[\sqrt{s}kv_{k+1}+r(x_{k+1}-x^\star)\right]  - \left[\sqrt{s}(k-1)v_k+r(x_k-x^\star)\right]  = -s(k+r)G_s(y_{k}). \label{eqn: im-vel-prox}
\end{align}
With~\eqref{eqn: implicit-velocity-phasespace} and~\eqref{eqn: nag-2-yx}, we can calculate the difference between $y_k$ and $y_{k-1}$ as
\begin{align}
y_k - y_{k-1}= \frac{k-1}{k+r}\cdot \sqrt{s}v_k - s G_s(y_{k-1}).     \label{eqn: fista-tran}
\end{align}

\item[\textbf{(2)}] Then, with the subscripts labeled, we calculate the iterative difference of the Lyapunov function~\eqref{eqn: lyapunov-fista-implicit-velocity} as
\begin{align*}
\mathcal{E}(k+1) -\mathcal{E}(k) &=  sk(k+r)\left(\Phi\left(x_{k+1}\right) - \Phi\left(x_{k}\right)\right) + s(2k+r+1)\left(\Phi\left(x_{k+1}\right) - \Phi(x^\star)\right) \\
& \mathrel{\phantom{=}} -  \left\langle s(k+r)G_s(y_k), \sqrt{s}(k-1)v_k + r(x_k - x^\star)\right\rangle   + \frac{s^2(k+r)^2}{2} \| G_s(y_k) \|^2 \\ 
 &=  sk(k+r)\left(\Phi\left(x_{k+1}\right) - \Phi\left(x_{k}\right)\right) + s(2k+r+1)\left(\Phi\left(x_{k+1}\right) - \Phi(x^\star)\right) \\
& \mathrel{\phantom{=}} -  \left\langle s(k+r)G_s(y_k), \frac{\sqrt{s}k(k-1)v_k}{k+r} + r(y_k - x^\star)\right\rangle   + \frac{s^2(k+r)^2}{2} \| G_s(y_k) \|^2 \\ 
&=  \underbrace{sk(k+r)\left(\Phi\left(x_{k+1}\right) - \Phi\left(x_{k}\right)\right)}_{\mathbf{I}_1} + \underbrace{s(2k+r+1)\left(\Phi\left(x_{k+1}\right) - \Phi(x^\star)\right)}_{\mathbf{I}_2} \\
& \mathrel{\phantom{=}} -\underbrace{ s^{\frac{3}{2}}k(k-1) \left\langle G_s(y_k), v_k \right\rangle}_{\mathbf{II}_1} - \underbrace{sr(k+r) \left\langle G_s(y_k), y_k - x^\star\right\rangle}_{\mathbf{II}_2} \\ & \mathrel{\phantom{=}}  + \frac{s^2(k+r)^2}{2} \| G_s(y_k) \|^2,
\end{align*}

\item[\textbf{(3)}] This step is also the essential difference with~\citep[Section 4.2]{chen2022gradient}. With~\Cref{lem: key-inequality}, we can obtain one inequality by putting $x_{k+1}$ and $x_k$ into~\eqref{eqn: key-inequality} as   
    \begin{align}
        \Phi(x_{k+1}) - \Phi(x_k)     &\le \langle G_s(y_k), y_k - x_k\rangle - \left(s - \frac{s^2 L}{2}\right)\|G_s(y_k)\|^2 \nonumber\\
                                      & = \langle G_s(y_k), y_k - y_{k-1} - \left( x_k - y_{k-1} \right)\rangle - \left(s - \frac{s^2 L}{2}\right)\|G_s(y_k)\|^2 \nonumber\\ 
                                      & = \frac{k-1}{k+r} \cdot \sqrt{s} \langle G_s(y_k), v_k\rangle - \left(s - \frac{s^2 L}{2}\right)\|G_s(y_k)\|^2  \label{eqn: fista-iv-1-key}
    \end{align}
where the last equality follows FISTA and the expression of $y_{k} - y_{k-1}$ in~\eqref{eqn: fista-tran}. Directly, with~\eqref{eqn: fista-iv-1-key}, we can estimate the difference between $\mathbf{I}_1$ and $\mathbf{II}_1$ as

\[
\mathbf{I}_1 - \mathbf{II}_1 \leq  - s^2k(k+r)\left(1 - \frac{s L}{2}\right)\|G_s(y_k)\|^2.
\]
Similarly, with~\Cref{lem: key-inequality}, we can obtain the other inequality  by putting $x_{k+1}$ and $x^\star$ into~\eqref{eqn: key-inequality} as 
\begin{equation*}
    \Phi(x_{k+1}) - \Phi(x^\star)  \le \langle G_s(y_k), y_k - x^\star\rangle - \left(s - \frac{s^2 L}{2}\right)\|G_s(y_k)\|^2 ,
\end{equation*}
which leads to the difference between $\mathbf{I}_2$ and $\mathbf{II}_2$ being eastimated as
\begin{align*}
\mathbf{I}_2 - \mathbf{II}_2 & =   s(2k+r+1)\left( \Phi(x_{k+1}) - \Phi(x^\star) \right) - sr(k+r)\left\langle G_s(y_{k}),  y_{k} - x^\star\right\rangle \\
                             & \leq  -  r(k+r)\left(s^2 - \frac{s^3 L}{2}\right)\|G_s(y_k)\|^2- s \left[ (r-2)k + r^2 - r - 1  \right]\left( \Phi(x_{k+1}) - \Phi(y^\star) \right).
\end{align*}
Furthermore, the iterative difference of the Lyapunov function~\eqref{eqn: lyapunov-fista-grad-correction} is computed as
\begin{align*}
\mathcal{E}(k+1) - \mathcal{E}(k) \leq & - \frac{s^2(k+r)^2}{2}\left(1- sL\right)\|G_s(y_k)\|^2 \nonumber \\
                                       & - s \left[ (r-2)k + r^2 - r - 1  \right]\left( \Phi(x_{k+1}) - \Phi(x^\star) \right),  
\end{align*}
which is exactly the same as~\eqref{eqn: final-iterative}. Therefore, the lines following~\eqref{eqn: final-iterative} complete the proof of~\Cref{thm: fista}.
\end{itemize}
\end{proof}

\section{Conclusion and discussion}
\label{sec: discussion}

In this paper, we generalize the gradient norm minimization to the composite optimization based on improving a pivotal inequality in~\citep[Lemma 2.3]{beck2009fast} and~\citep[(22)]{su2016differential}. With the well-constructed Lyapunov function, we use the phase-space representation, whatever gradient-correction or implicit-velocity, to obtain that the squared proximal subgradient norm of ISTA converges at an inverse square rate and the squared proximal subgradient norm of FISTA is accelerated to converge at an inverse cubic rate. Furthermore, we highlight some merits of the proximal subgradient norm minimization in practice. Meanwhile, we also find that the model used to characterize the linear inverse problem with sparse representation is the composite objective function, a smooth convex function plus a continuous convex function. However, the symmetric matrix $A^{T}A$ in the least square part is invertible, even ill-conditioned. In other words, there always exists a maximum and a minimum in the spectrum of the symmetric matrix $A^{T}A$. Furthermore, it is more reasonable to use a model with a quadratic function plus a continuous convex function to characterize the linear inverse problem with sparse representation, to say the least, a model with a strongly-convex function plus a continuous convex function. In further works, we will continue to investigate the convergence rate of ISTA and FISTA.   

The distinctive character of the $\ell_1$-norm is to capture the sparse structure, which is widely used in the statistics, such as the Lasso model and its variants. Based on the Lasso model, the sorted $\ell_1$-penalized estimation (SLOPE) is proposed to investigate the asymptotic minimax problem~\citep{su2016slope}, which is adaptive to unknown sparsity. Furthermore, the relation between false discoveries and the Lasso path is investigated in~\citep{su2017false}. An interesting direction is to investigate further the theory in statistics related to Lasso based on understanding and improving composite optimization algorithms. Meanwhile, we also point out that the differential inclusion used to sparse recovery~\citep{osher2016sparse} is also related to composite optimization, which is probably worth being further considered in the future.

{\small
\bibliographystyle{abbrvnat}
\bibliography{reference}
}
\end{document}